\documentclass[3p,12pt]{elsarticle}

\usepackage{amssymb,latexsym,amsmath,amsthm}
\journal{}

\newcommand{\set}[1]{\left\{#1\right\}}
\newcommand{\abs}[1]{\left|#1\right|}
\newcommand{\p}{\partial}

\newcommand{\mP}{\mathbf{P}}
\newcommand{\mU}{\mathbf{U}}
\newcommand{\mV}{\mathbf{V}}
\newcommand{\mc}{\mathbf{c}}

\newcommand{\mf}{\mathbf{f}}

\newcommand{\mn}{\mathbf{n}}

\newcommand{\mx}{\mathbf{x}}
\newcommand{\my}{\mathbf{y}}
\newcommand{\mz}{\mathbf{z}}
\newcommand{\vg}{\boldsymbol{\gamma}}
\newcommand{\vn}{\boldsymbol{\nu}}

\newcommand{\vt}{\boldsymbol{\theta}}

\newcommand{\vx}{\boldsymbol{\xi}}

\newtheorem{thm}{Theorem}[section]

\newtheorem{lem}[thm]{Lemma}

\begin{document}

\begin{frontmatter}



\title{Certain properties of MUSIC-type imaging functional in inverse scattering from an open, sound-hard arc}
%\tableofcontents

\author{Won-Kwang Park}
\ead{parkwk@kookmin.ac.kr}
\address{Department of Mathematics, Kookmin University, Seoul, 136-702, Korea.}

\begin{abstract}
  This paper concerns mathematical formulation of well-known MUltiple SIgnal Classification (MUSIC)-type imaging functional in the inverse scattering problem by an open sound-hard arc. Based on the physical factorization of so-called Multi-Static Response (MSR) matrix and the structure of left-singular vectors liked to the non-zero singular values of MSR matrix, we construct a relationship between imaging functional and Bessel function of order $1$ of the first kind. We then expound certain properties of MUSIC and present numerical results for a number of differently chosen smooth arcs.
\end{abstract}


\end{frontmatter}





\section{Introduction}
The main purpose of this paper is to identify mathematical structure of so-called MUltiple SIgnal Classification (MUSIC)-type imaging functional for two-dimensional perfectly conducting, arc-like smooth cracks at a fixed frequency. From our best knowledge, the mathematical modeling of time-harmonic inverse scattering problem from a sound-hard open arc has been considered in \cite{M2}. In this remarkable paper, a boundary integral equation approach has been concerned for solving the direct scattering problem a complete description of a numerical solution method including a rigorous convergence and error analysis is considered. Then, the inverse scattering problem related to retrieval of the sound-hard crack based on  Newton-type iterative scheme has been investigated in \cite{M1}. After the successful application, in many works \cite{AS1,KS1,L2,L1}, many authors have proposed various algorithms, most of which are based on Newton-type iteration schemes.

Generally, for a successful application of Newton-type iterative based scheme, a good initial guess that is close to the unknown crack must be applied in the beginning of iteration procedure in order to avoid the non-convergence phenomenon. For this reason, various non-iterative shape reconstruction algorithms. Specifically, MUltiple SIgnal Classification (MUSIC) has been shown to be feasible in various inverse scattering problems and generalized to the shape reconstruction of various kind of targets in two- and three-dimensional problems. Related works can be found in \cite{AJP,AGKPS,AIL1,AKLP,CZ,HSZ1,P-MUSIC1,PL1,PL3,S2,ZC} and references therein. In recent works \cite{JKP,JP2}, the structure of MUSIC was identified in full- and limited-view inverse scattering problem for imaging of sound-soft arc (Dirichlet boundary condition -- Transverse Magnetic polarization) but as to the sound-soft arc, little has been theoretically studied.

Motivated by the above, this paper concerns the identification of some properties of the MUSIC-type imaging algorithm in full-view inverse scattering from sound-soft arcs at a fixed single frequency. For this purpose, we investigate a relationship between the MUSIC-type imaging function and the Bessel function of order $1$ of the first kind in order to identify various properties of MUSIC.

This paper is constructed as follows. In section \ref{sec:2}, we briefly survey the two-dimensional direct scattering problem from sound-soft arc and introduce MUSIC-type imaging algorithm. In section \ref{sec:3}, we identify the structure of the MUSIC-type imaging function, and discover its certain properties. Some numerical experiments are shown in section \ref{sec:4} in order to support discovered properties. In section \ref{sec:5}, a short conclusion is presented.

\section{Direct scattering problem and MUSIC algorithm}\label{sec:2}
\subsection{Two-dimensional direct scattering problem from sound-hard arc}
We briefly introduce the two-dimensional direct scattering problem by a perfectly conducting crack, denoted by $\Gamma$. We recommend \cite{M1} for a detailed discussion. Suppose that $\Gamma$ is an oriented piecewise-smooth non-intersecting arc without a cusp that can be represented as
\begin{equation}\label{Crack}
  \Gamma:=\set{\vg(s):-1\leq s\leq1},
\end{equation}
where $\vg\in\mathcal{C}^{\infty}$ is injective.

Let $u(\mx;\vt)\in\mathcal{C}^2(\mathbb{R}^2\backslash\Gamma)\cap\mathcal{C}(\overline{\mathbb{R}^2\backslash\Gamma})$, which is continuous at $\vg(-1)$ and $\vg(1)$, be the total field that satisfying the following Helmholtz equation
\begin{equation}\label{Problem1}
  \triangle u(\mx;\vt)+k^2u(\mx;\vt)=0\quad\mbox{in}\quad\mathbb{R}^2\backslash\Gamma
\end{equation}
with following Neumann boundary condition
\begin{equation}\label{Problem2}
  \frac{\p u(\mx;\vt)}{\p\mn(\mx)}=0\quad\mbox{on}\quad\Gamma\backslash\set{\vg(-1),\vg(1)},
\end{equation}
where $\mn(\mx)$ is a unit normal to $\Gamma$ at $\mx$.

It is well-known that $u(\mx;\vt)$ can be decomposed as $u(\mx;\vt)=u_i(\mx;\vt)+u_s(\mx;\vt)$, where $u_i(\mx;\vt):=e^{ik\vt\cdot\mx}$ be the given plane-wave incident field with direction $\vt\in\mathbb{S}^1$ and strictly positive wave number $k$ and $u_s(\mx;\vt)$ be the unknown scattered field, which is demanded to satisfy the Sommerfeld radiation condition
\[\lim_{\abs{\mx}\to\infty}\sqrt{\abs{\mx}}\left(\frac{\p u_s(\mx;\vt)}{\p\abs{\mx}}-iku_s(\mx;\vt)\right)=0\]
uniformly into all directions $\hat{\mx}=\frac{\mx}{\abs{\mx}}$. Note that based on \cite{M1}, we do not require any edge condition for the behavior of $u(\mx;\vt)$ at $\vg(-1)$ and $\vg(1)$.

On the basis of result in \cite{M1}, $u_s(\mx;\vt)$ can be expressed as the following double-layer potential
\begin{equation}\label{DLP}
u_s(\mx;\vt)=\int_{\Gamma}\frac{\p\Phi(\mx,\my)}{\p\vn(\my)}\psi(\my,\vt)d\my\quad\mbox{for}\quad\mx\in\mathbb{R}^2\backslash\Gamma,
\end{equation}
where $-\psi(\mx,\vt)=u_{+}(\mx,\vt;k)-u_{-}(\mx,\vt;k)$, $\Phi(\mx,\my)$ is the two-dimensional fundamental solution to the Helmholtz equation
\[\Phi(\mx,\my):=\frac{i}{4}\mathrm{H}_0^1(k\abs{\mx-\my})\quad\mbox{for}\quad\mx\ne\my,\]
and $\mathrm{H}_0^1$ denotes the Hankel function of order zero and of the first kind.

Suppose that for all $\mx\in\Gamma\backslash\set{\vg(-1),\vg(1)}$, the limit of the following quantity exists:
\[\frac{\p u_{\pm}(\mx,\vt)}{\p\mn(\mx)}=\lim_{h\to+0}\nabla u(\mx\pm h\mn(\mx),\vt)\cdot\mn(\mx).\]

The far-field pattern $u_{\infty}(\hat{\mx},\vt;k)$ of the scattered field $u_{\mathrm{scat}}(\mx,\vt;k)$ is defined on the two-dimensional unit circle $\mathbb{S}^1$. It can be represented as
\[u_s(\mx,\vt)=\frac{e^{ik\abs{\mx}}}{\sqrt{\abs{\mx}}}\left\{u_{\infty}(\hat{\mx},\vt;k)+O\left(\frac{1}{\abs{\mx}}\right)\right\}\]
uniformly in all directions $\hat{\mx}=\mx/\abs{\mx}$ and $\abs{\mx}\longrightarrow\infty$. From the above representation and the asymptotic formula for the Hankel function, the far-field pattern can be written as
\begin{align}
\begin{aligned}\label{FFPN}
u_\infty(\hat{\mx},\vt;k)&=-\frac{e^{i\frac{\pi}{4}}}{\sqrt{8\pi k}}\int_{\Gamma}\frac{\p e^{-ik\hat{\mx}\cdot\my}}{\p\mn(\my)}\bigg(u_{+}(\my,\vt;k)-u_{-}(\my,\vt;k)\bigg)d\my\\
&=-\sqrt{\frac{k}{8\pi}}e^{-i\frac{\pi}{4}} \int_{\Gamma}\hat{\mx}\cdot\mn(\my)e^{-ik\hat{\mx}\cdot\my}\psi(\my,\vt;k)d\my.
\end{aligned}
\end{align}

\subsection{Introduction to MUSIC-type imaging}
We apply the far-field pattern formula (\ref{FFPN}) to introduce MUSIC-type imaging functional. Before starting, we assume that the crack is divided into $M$ different segments of size of the order of half the wavelength $\lambda/2$. With respect to the Rayleigh resolution limit, any detail less than one-half of the wavelength cannot be probed, and only one point, say $\my_m$ for $m=1,2,\cdots,M$, at each segment is expected to contribute at the image space of the response matrix $\mathbb{K}(k)$ \cite{ABC,AKLP,PL1,PL3}. Now, let us consider the eigenvalue structure of the MSR matrix, whose element is the collected far-field at observation number $j$ for the incident number $l$:
\[\mathbb{K}:=\bigg[K_{jl}(\hat{\mx}_j,\vt_l;k)\bigg]_{j,l=1}^{N}=\bigg[u_\infty(\hat{\mx}_j,\vt_l)\bigg]_{j,l=1}^{N}.\]
In this paper we assume that $\hat{\mx}_j=-\vt_j$ for $j=1,2,\cdots,N$, i.e., under the the coincide configuration of incident and observation directions, the MSR matrix $\mathbb{K}$ is complex symmetric. Therefore, $\mathbb{K}$ can be decomposed as
\begin{equation}\label{MSRN}
\mathbb{K}=\sqrt{\frac{k}{8\pi}}e^{-i\frac{\pi}{4}}\int_{\Gamma}\mathbb{P}_{\mathrm{N}}(\hat{\mx},\my)\mathbb{Q}_{\mathrm{N}}(\hat{\mx},\my)^Td\my,
\end{equation}
where $\mathbb{P}_{\mathrm{N}}(\hat{\mx},\my)$ is the illumination vector
\begin{align}
\begin{aligned}\label{VecEN}
\mathbb{P}_{\mathrm{N}}(\hat{\mx},\my)&=-\bigg[\hat{\mx}_1\cdot\mn(\my)e^{-ik\hat{\mx}_1\cdot \my},\hat{\mx}_2\cdot\mn(\my)e^{-ik\hat{\mx}_2\cdot \my},\cdots,\hat{\mx}_N\cdot\mn(\my)e^{-ik\hat{\mx}_N\cdot\my}\bigg]^T\bigg|_{\hat{\mx}_j=-\vt_j}\\
&=\bigg[\vt_1\cdot\mn(\my)e^{ik\vt_1\cdot\my},\vt_2\cdot\mn(\my)e^{ik\vt_2\cdot \my},\cdots,\vt_N\cdot\mn(\my)e^{ik\vt_N\cdot\my}\bigg]^T
\end{aligned}
\end{align}
and where $\mathbb{Q}_{\mathrm{N}}(\hat{\mx},\my)$ is the corresponding density vector
\begin{equation}
\mathbb{Q}_{\mathrm{N}}(\hat{\mx},\my)=\bigg[\psi(\my,\vt_1),\psi(\my,\vt_2),\cdots,\psi(\my,\vt_N)\bigg]^T.
\end{equation}

Formula (\ref{MSRN}) is the factorization of the MSR matrix that separates the known incoming plane-wave information from the unknown information similar to the Dirichlet boundary condition case. The range of $\mathbb{K}$ is determined by the span of the $\mathbb{P}_{\mathrm{N}}(\hat{\mx},\my)$ corresponding to $\Gamma$. Hence, on the basis of the result in \cite{HSZ1}, by applying a set of remaining singular vectors of $\mathbb{K}$, a signal subspace can be defined.

Let the singular value decomposition of the matrix $\mathbb{K}$ be
\[\mathbb{K}=\mathbb{UBV}^*=\sum_{m=1}^{M}\sigma_m\mU_m\mV_m^*,\]
where $\mU_m\in\mathbb{C}^{N\times1}$ are the left-singular vectors of $\mathbb{K}$ and $\sigma_m$ are nonnegative real-valued singular values such that
\[\sigma_1\geq\sigma_2\geq\cdots\geq\sigma_M>0\quad\mbox{and}\quad\sigma_j=0\quad\mbox{for}\quad j=M+1,M+2\cdots,N.\]
Alternatively, $\sigma_j$, for $j=M+1,M+2,\cdots,N$, could merely be very small, below the noise level of the system represented by $\mathbb{K}$. Then, the first $M$ columns of the matrix $\mathbb{U}$, $\set{\mU_1,\mU_2,\cdots,\mU_M}$, provide an orthonormal basis for $\mathbb{K}$ and the rest of the columns, $\set{\mU_{M+1},\mU_{M+2},\cdots,\mU_N}$, provides a basis for the null (or noise) space of $\mathbb{K}$. So, one can form the projection onto the null (or noise) subspace: this projection is given explicitly by
\begin{equation}
\mP_{\mbox{\tiny noise}}=\mathbb{I}-\sum_{m=1}^{M}\mU_m\mU_m^*,
\end{equation}
where $\mathbb{I}$ denotes $N\times N$ identity matrix.

On the basis of $\mathbb{P}_{\mathrm{N}}(\hat{\mx},\my)$ in (\ref{VecEN}), for any $\mz\in\mathbb{R}^2$ and $\mc_n\in\mathbb{R}^2\backslash\set{\mathbf{0}}$, define $\mf(\mz)\in\mathbb{C}^{N\times1}$ as
\begin{equation}\label{TestVectorN1}
\mf(\mz)=\bigg[(\mc_1\cdot\vt_1)e^{ik\vt_1\cdot\mz},(\mc_2\cdot\vt_2)e^{ik\vt_2\cdot\mz},\cdots,(\mc_N\cdot\vt_N)e^{ik\vt_N\cdot\mz}\bigg]^T.
\end{equation}
Then, there exists $N_0\in\mathbb{N}$ such that for any $N\geq N_0$, the following statement holds \cite{AK2}:
\[\mf(\mz)\in\mbox{Range}(\mathbb{K})\quad\mbox{if and only if}\quad\mz\in\set{\my_1,\my_2,\cdots,\my_M}.\]
This means that if a point $\mz$ satisfies $\mz\in\set{\my_1,\my_2,\cdots,\my_M}$ then $|\mP_{\mbox{\tiny noise}}(\mf(\mz))|=0$. Thus, an image of $\my_m\in\Gamma$, $m=1,2,\cdots,M$, can be obtained from computing
\begin{equation}\label{MUSICfunction}
\mathbb{W}(\mz)=\frac{1}{|\mP_{\mbox{\tiny noise}}(\mf(\mz))|}.
\end{equation}
The resulting plot of $\mathbb{W}(\mz)$ is expected to exhibit peaks of large (theoretically, $+\infty$) magnitude at the $\my_m\in\Gamma$.

\section{Mathematical structure and intrinsic properties of imaging functional}\label{sec:3}
On the basis of the results in \cite{PL1}, the selection of $\mathbf{c}_n$ is a strong prerequisite. The selection depends on the shape of the supporting curve $\Gamma$. Roughly speaking, $\mathbf{c}_n$ must be of the form $\mn(\mx_m)$ for $m=1,2,\cdots,M$. Unfortunately, we have no \textit{a priori} information of shape of $\Gamma$. Due to this reason, in \cite{HSZ1,PL1}, a large number of directions are applied in order to find an optimal vector $\mathbf{c}_n$ but this process requires large computational costs. Hence, motivated from recent work \cite{P-MUSIC1,P-SUB3}, we assume that $\mathbf{c}_n$ satisfies $\mathbf{c}_n\cdot\vt_n=1$ for all $n$. Correspondingly, instead of (\ref{TestVectorN1}), we apply
\begin{equation}\label{VecF}
  \mf(\mz)=\frac{1}{\sqrt{N}}\bigg[e^{ik\vt_1\cdot\mz},e^{ik\vt_2\cdot\mz},\cdots,e^{ik\vt_N\cdot\mz}\bigg]^T,
\end{equation}
and explore some properties of MUSIC-type imaging algorithm.

\subsection{Relationship with Bessel function of first order}
Before starting, we recall a useful relationship, which plays a key roll of our identification.
\begin{lem}\label{TheoremBessel}
  Assume that $\set{\vt_n:n=1,2,\cdots,N}$ spans $\mathbb{S}^1$. Then, for sufficiently large $N$, $\vx\in\mathbb{S}^1$, and $\mx\in\mathbb{R}^2$.
\[\frac{1}{N}\sum_{n=1}^{N}(\vt_n\cdot\vx)e^{ik\vt_n\cdot\mx}=\frac{1}{2\pi}\int_{\mathbb{S}^1}(\vt\cdot\vx)e^{ik\vt\cdot\mx}dS(\vt)=i\left(\frac{\mx}{|\mx|}\cdot\vx\right)J_1(k|\mx|),\]
where $J_\nu$ denotes the Bessel function of integer order $\nu$ of the first kind.
\end{lem}

Now, we state the main result.
\begin{thm}\label{TheoremMUSIC}
  Let $N>M$. Then, for sufficiently large $N$ and $k$, (\ref{MUSICfunction}) can be written as follows:
  \begin{equation}\label{StructureMUSIC}
    \mathbb{W}(\mz)=\left(1-2\sum_{m=1}^{M}\left(\frac{\mz-\my_m}{|\mz-\my_m|}\cdot\mn(\my_m)\right)^2J_1(k|\mz-\my_m|)^2\right)^{-1/2}.
\end{equation}
\end{thm}
\begin{proof}
Let us consider the polar-coordinate representation: for $\vt,\vx\in\mathbb{S}^1$, $\vt=[\cos\theta,\sin\theta]^T$, and $\vx=[\cos\psi,\sin\psi]^T$. Since
\begin{equation}\label{NormofSingularVector}
  \sum_{n=1}^{N}(\vt_n\cdot\vx)^2\approx\frac{N}{2\pi}\int_{\mathbb{S}^1}(\vt\cdot\vx)^2d\vt=\frac{N}{2\pi}\int_0^{2\pi}\cos^2(\phi-\psi)d\phi=\frac{N}{2},
\end{equation}
the left singular vectors are of the form (see \cite{AGKPS})
\[\mU_m\approx\sqrt{\frac{2}{N}}\bigg[(\vt_1\cdot\mn(\my_m))e^{ik\vt_1\cdot\my_m},
(\vt_2\cdot\mn(\my_m))e^{ik\vt_2\cdot\my_m},\cdots,(\vt_N\cdot\mn(\my_m))e^{ik\vt_N\cdot\my_m}\bigg]^T.\]

Since, we select $\mf(\mz)$ as (\ref{VecF}), $\mP_{\mbox{\tiny noise}}$ can be written as
\begin{align*}
  &\mP_{\mbox{\tiny noise}}(\mf(\mz)) =\left(\mathbb{I}_{N}-\sum_{m=1}^{M}\mU_m\overline{\mU}_m^T\right)\mf(\mz)\\
  &\approx\frac{1}{\sqrt{N}}\left[\begin{array}{c}
        e^{ik\vt_1\cdot\mz} \\
        e^{ik\vt_2\cdot\mz} \\
        \vdots \\
        e^{ik\vt_N\cdot\mz} \\
      \end{array}\right]-
      \frac{2}{N\sqrt{N}}\sum_{m=1}^{M}\left[\begin{array}{c}
      \displaystyle(\vt_1\cdot\mn(\my_m))e^{ik\vt_1\cdot\my_m}\sum_{n=1}^{N}(\vt_n\cdot\mn(\my_m))e^{ik\vt_n\cdot(\mz-\my_m)}\\
      \displaystyle(\vt_2\cdot\mn(\my_m))e^{ik\vt_2\cdot\my_m}\sum_{n=1}^{N}(\vt_n\cdot\mn(\my_m))e^{ik\vt_n\cdot(\mz-\my_m)}\\
      \vdots\\
      \displaystyle(\vt_N\cdot\mn(\my_m))e^{ik\vt_N\cdot\my_m}\sum_{n=1}^{N}(\vt_n\cdot\mn(\my_m))e^{ik\vt_n\cdot(\mz-\my_m)}
      \end{array}\right]\\
      &=\frac{1}{\sqrt{N}}\left[\begin{array}{c}
      \displaystyle e^{ik\vt_1\cdot\mz}-2i\sum_{m=1}^{M}(\vt_1\cdot\mn(\my_m))\left(\frac{\mz-\my_m}{|\mz-\my_m|}\cdot\mn(\my_m)\right)e^{ik\vt_1\cdot \my_m}J_1(k|\mz-\my_m|)\\
      \displaystyle e^{ik\vt_2\cdot\mz}-2i\sum_{m=1}^{M}(\vt_2\cdot\mn(\my_m))\left(\frac{\mz-\my_m}{|\mz-\my_m|}\cdot\mn(\my_m)\right)e^{ik\vt_2\cdot \my_m}J_1(k|\mz-\my_m|)\\
      \vdots\\
      \displaystyle e^{ik\vt_N\cdot\mz}-2i\sum_{m=1}^{M}(\vt_N\cdot\mn(\my_m))\left(\frac{\mz-\my_m}{|\mz-\my_m|}\cdot\mn(\my_m)\right)e^{ik\vt_N\cdot \my_m}J_1(k|\mz-\my_m|)
      \end{array}\right].
\end{align*}
Hence, we can obtain
\[|\mP_{\mbox{\tiny noise}}(\mf(\mz))|=\left(\frac{1}{N}\sum_{n=1}^{N}\bigg(1+\Psi_1-\overline{\Psi}_1+\Psi_2\overline{\Psi}_2\bigg)\right)^{1/2},\]
where
\begin{align*}
  \Psi_1&=2i\sum_{m=1}^{M}(\vt_n\cdot\mn(\my_m))\left(\frac{\mz-\my_m}{|\mz-\my_m|}\cdot\mn(\my_m)\right) e^{ik\vt_n\cdot(\mz-\my_m)}J_1(k|\mz-\my_m|)\\
  \Psi_2&=2\sum_{m=1}^{M}(\vt_n\cdot\mn(\my_m))\left(\frac{\mz-\my_m}{|\mz-\my_m|}\cdot\mn(\my_m)\right)e^{ik\vt_n\cdot \my_m}J_1(k|\mz-\my_m|).
\end{align*}

Since
\begin{align*}
  \sum_{n=1}^{N}\Psi_1&=2i\sum_{n=1}^{N}\sum_{m=1}^{M}(\vt_n\cdot\mn(\my_m))\left(\frac{\mz-\my_m}{|\mz-\my_m|}\cdot\mn(\my_m)\right) e^{ik\vt_n\cdot(\mz-\my_m)}J_1(k|\mz-\my_m|)\\
  &=2i\sum_{m=1}^{M}\left(\frac{\mz-\my_m}{|\mz-\my_m|}\cdot\mn(\my_m)\right)\sum_{n=1}^{N}\bigg((\vt_n\cdot\mn(\my_m))e^{ik\vt_n\cdot(\mz-\my_m)}\bigg)J_1(k|\mz-\my_m|)\\
  &=-2N\sum_{m=1}^{M}\bigg(\frac{\mz-\my_m}{|\mz-\my_m|}\cdot\mn(\my_m)\bigg)^2J_1(k|\mz-\my_m|)^2,
\end{align*}
we can obtain
\begin{equation}\label{term3}
    \frac{1}{N}\sum_{n=1}^{N}(\Psi_1-\overline{\Psi}_1)=-4\sum_{m=1}^{M}\bigg(\frac{\mz-\my_m}{|\mz-\my_m|}\cdot\mn(\my_m)\bigg)^2J_1(k|\mz-\my_m|)^2.
\end{equation}

Furthermore, we can evaluate
\begin{align*}
  \sum_{n=1}^{N}\Psi_2\overline{\Psi}_2=&4\sum_{n=1}^{N}\left(\sum_{m=1}^{M}(\vt_n\cdot\mn(\my_m))\left(\frac{\mz-\my_m}{|\mz-\my_m|}\cdot\mn(\my_m)\right)e^{ik\vt_n\cdot \my_m}J_1(k|\mz-\my_m|)\right)\\
  &\times\left(\sum_{m'=1}^{M}(\vt_n\cdot\mn(\my_{m'}))\left(\frac{\mz-\my_{m'}}{|\mz-\my_{m'}|}\cdot\mn(\my_{m'})\right)e^{-ik\vt_n\cdot\my_{m'}}J_1(k|\mz-\my_{m'}|)\right)\\
  =&4\sum_{m=1}^{M}\sum_{n=1}^{N}\left((\vt_n\cdot\mn(\my_m))^2\left(\frac{\mz-\my_m}{|\mz-\my_m|}\cdot\mn(\my_m)\right)^2J_1(k|\mz-\my_m|)^2\right)\\
  =&4N\sum_{m=1}^{M}\left(\frac{1}{N}\sum_{n=1}^{N}(\vt_n\cdot\mn(\my_m))^2\right)\left(\frac{\mz-\my_m}{|\mz-\my_m|}\cdot\mn(\mx)\right)^2J_1(k|\mz-\my_m|)^2\\
  =&2N\sum_{m=1}^{M}\bigg(\frac{\mz-\my_m}{|\mz-\my_m|}\cdot\mn(\my_m)\bigg)^2J_1(k|\mz-\my_m|)^2.
\end{align*}

Hence, we can conclude that
\begin{equation}\label{term4}
  \frac{1}{N}\sum_{n=1}^{N}\Psi_2\overline{\Psi}_2=2\sum_{m=1}^{M}\bigg(\frac{\mz-\my_m}{|\mz-\my_m|}\cdot\mn(\my_m)\bigg)^2J_1(k|\mz-\my_m|)^2.
\end{equation}

Therefore, by (\ref{term3}) and (\ref{term4}), we can obtain
\[|\mP_{\mbox{\tiny noise}}(\mf(\mz))|=\left\{1-2\sum_{m=1}^{M}\left(\frac{\mz-\my_m}{|\mz-\my_m|}\cdot\mn(\my_m)\right)^2J_1(k|\mz-\my_m|)^2\right\}^{1/2}.\]
With this, we can derive (\ref{StructureMUSIC}). This completes the proof.
\end{proof}

\subsection{Intrinsic properties of imaging functional}
Based on the structure (\ref{StructureMUSIC}), we can observe following some intrinsic properties of MUSIC-type imaging functional.

\begin{enumerate}\renewcommand{\theenumi}{(P\arabic{enumi})}
\item\label{P1} Total number of incident and observation directions $N$ must be sufficiently large. Furthermore, based on the relationship between imaging functional and Bessel function, imaging result is highly depending on the applied wavenumber. So, $k$ must be large enough. These assumptions in Theorem \ref{TheoremMUSIC} are very strong \textit{a prior} conditions, refer to the simulation results in Section \ref{sec:4}.
\item\label{P2} Since $J_1(1.8412)'\approx0$, $J_1$ has its maximum value $0.5819$ at $x\approx1.8412$. Therefore,
    \[2\left(\frac{\mz-\my_m}{|\mz-\my_m|}\cdot\mn(\my_m)\right)^2J_1(k|\mz-\my_m|)^2\leq2(0.5819)^2=0.6772\ne1.\]
    This means that there is no blow-up of $\mathbb{W}(\mz)$.
\item\label{P3} On the basis of the shape of $J_1(x)^2$ (see Figure \ref{FigureBessel}), instead of true shape, two curves will appear along the normal direction in the neighborhood of $\Gamma$. Related discussion can be found in \cite{HSZ1,P-MUSIC1,P-SUB3} and numerical results in Section \ref{sec:4}. Note that following asymptotic form holds
\[J_1(k|\mz-\my_m|)^2\approx\left(\frac{k|\mz-\my_m|}{2}\right)^2\quad\mbox{for}\quad0<k|\mz-\my_m|\ll\sqrt{2},\]
it is expected that true shape of $\Gamma$ can be imaged via the map of $\mathbb{W}(\mz)$ when $k\longrightarrow+\infty$. However, this is an ideal assumption.
\item Following recent work \cite{CZ}, it is expected that vectors $\mc_n$, $n=1,2,\cdots,N$, in (\ref{TestVectorN1}) can be estimated. Then, the structure of $\mathbb{W}(\mz)$ becomes (see \cite{P-MUSIC1} for instance)
\[\mathbb{W}(\mz)\approx\left(1-\sum_{m=1}^{M}J_0(k|\mz-\my_m|)^2\right)^{-1/2}\]
correspondingly, almost true shape of $\Gamma$ can be obtained. Related results are exhibited in \cite{PL1}.
\item Based on recent works \cite{P-SUB3,AGKPS,P-SUB1}, applying multi-frequency improves the imaging performance. However, in the imaging of sound-hard arc, application of multi-frequency will reduce the artifacts but true shape of crack cannot be retrieved.
\end{enumerate}

\begin{figure}
\begin{center}
\includegraphics[width=0.9\textwidth]{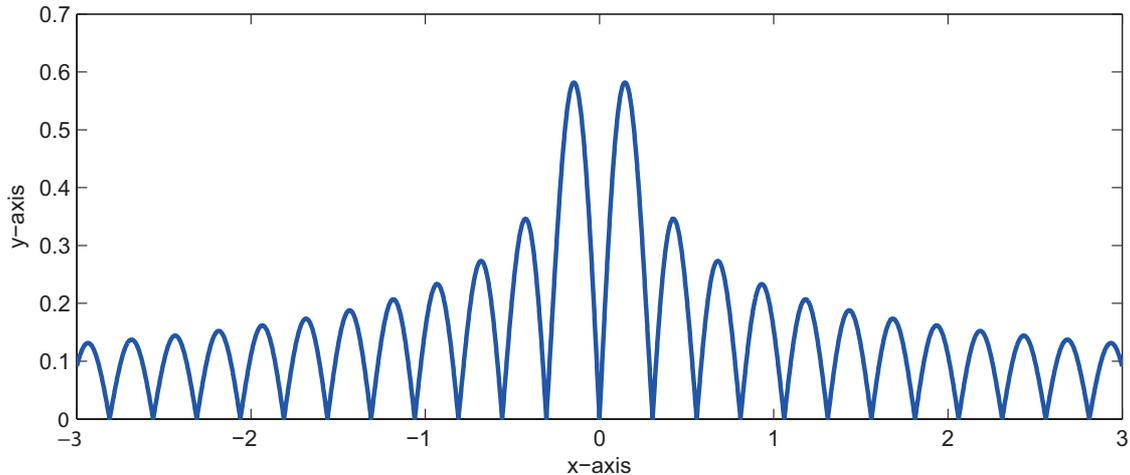}
\caption{$1-$D plot of $|J_1(kx)|$ for $k=2\pi/0.5$.}\label{FigureBessel}
\end{center}
\end{figure}

\section{Simulation results and discussions}\label{sec:4}
In this section, some results of numerical simulations with noisy data are exhibited for supporting our result in Theorem \ref{TheoremMUSIC}. Throughout this section, two curves $\Gamma_j$ are chosen to describe the sound-hard arc such that
\begin{align*}
  \Gamma_1&=\set{[s-0.2,-0.5s^2+0.4]:-0.5\leq s\leq0.5},\\
  \Gamma_2&=\set{[s+0.2,s^3+s^2-0.4]:-0.5\leq s\leq0.5}.
\end{align*}
Total number of incident and observation directions is set to $N=32$ and incident vectors $\vt_l$ are selected as
\[\vt_l=-\left[\cos\frac{2\pi(l-1)}{N-1},\sin\frac{2\pi(l-1)}{N-1}\right]^T.\]

It is worth mentioning that, since the reliable and efficient solution of the direct scattering problem indicated previously is very important, elements $u_{\infty}(\theta_j,\theta_l)$ for $j,l=1,2,\cdots,N$ of the $\mathbb{K}$ are generated by solving second-kind Fredholm integral equation along the crack, refer to \cite[Section 3]{N}. After obtaining the dataset, $20-$dB white Gaussian random noise is added to the unperturbed data via the MATLAB subroutine \texttt{awgn}. In order to perform the singular value decomposition of $\mathbb{K}$, MATLAB subroutine \texttt{svd} is applied. To obtain the number of nonzero singular values, a $0.01$-threshold scheme (select the first $M$-values $\sigma_m$ that $\sigma_m/\sigma_1\geq0.01$) is adopted. A more detailed discussion of threshold can be found in \cite{PL1,PL3} (see \cite{HSZ1} for the volumetric-extended target case).

First, let us consider the imaging of $\Gamma_1$. Figure \ref{Result1} exhibits the maps of $\mathbb{W}(\mz)$ for $\lambda=\pi$, $0.8$, $0.4$, and $0.2$. It is easy to observe that two ghost replicas with large magnitude in the neighborhood of $\Gamma_1$ so that observation \ref{P3} holds. The shape of $\Gamma_1$ cannot be identified via the map of $\mathbb{W}(\mz)$ if applied value of $k$ is small. In contrast, if $k$ is sufficiently large, one can identify two curves. This supports the observation \ref{P1}. Furthermore, if $k$ becomes large, identified shape is close to the true shape of $\Gamma_1$. Hence, although it is an ideal assumption, if $k=+\infty$, true shape of $\Gamma_1$ can be recognized via the map of $\mathbb{W}(\mz)$, refer to \ref{P3}. Notice that for selected values of $k$, maximum value of $\mathbb{W}(\mz)$ is not so large enough. Hence, there is no blow-up of $\mathbb{W}(\mz)$ and this supports observation \ref{P2}.

\begin{figure}
\begin{center}
\includegraphics[width=0.49\textwidth]{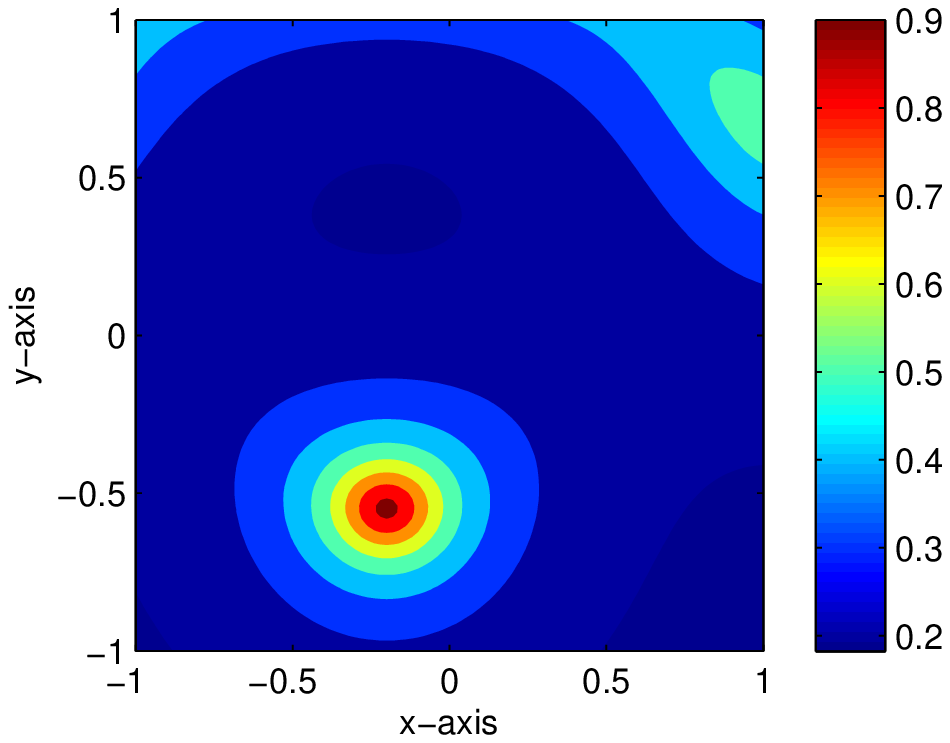}
\includegraphics[width=0.49\textwidth]{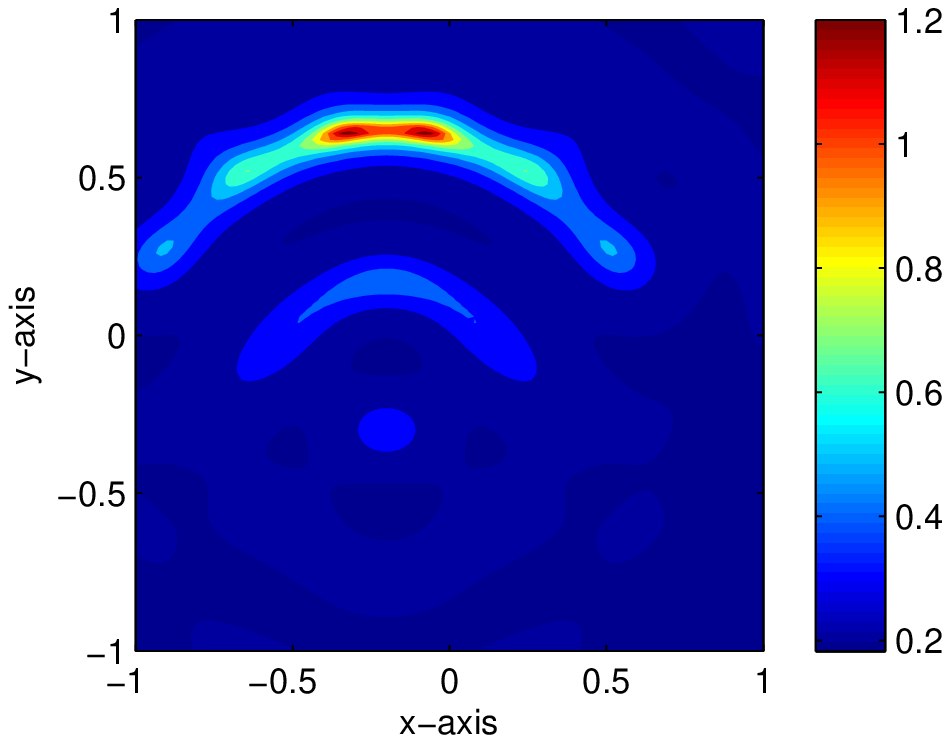}\\
\includegraphics[width=0.49\textwidth]{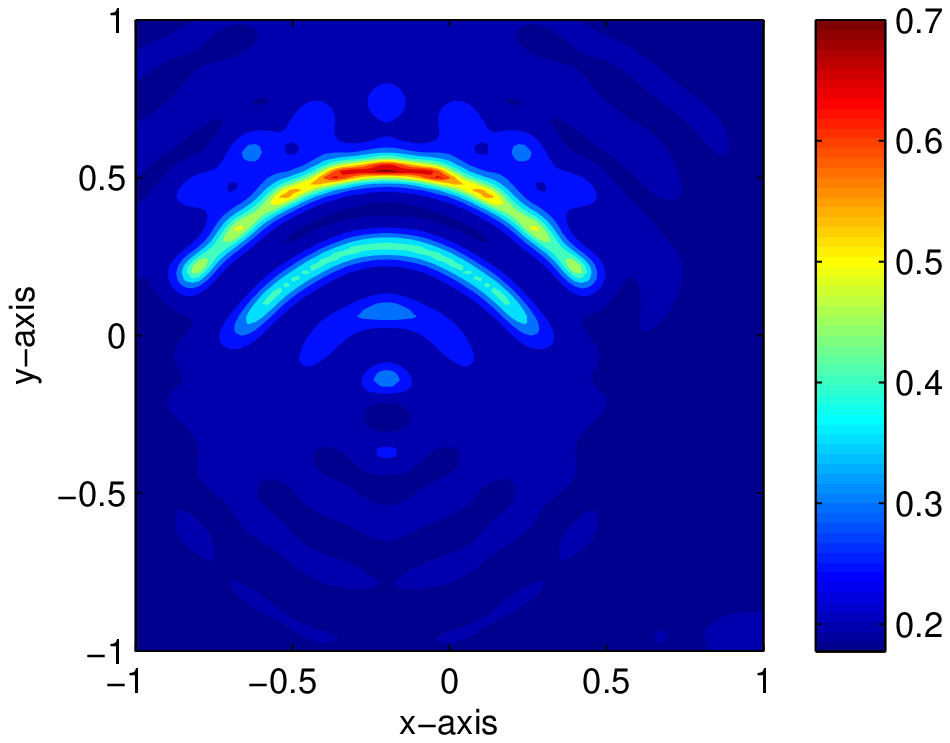}
\includegraphics[width=0.49\textwidth]{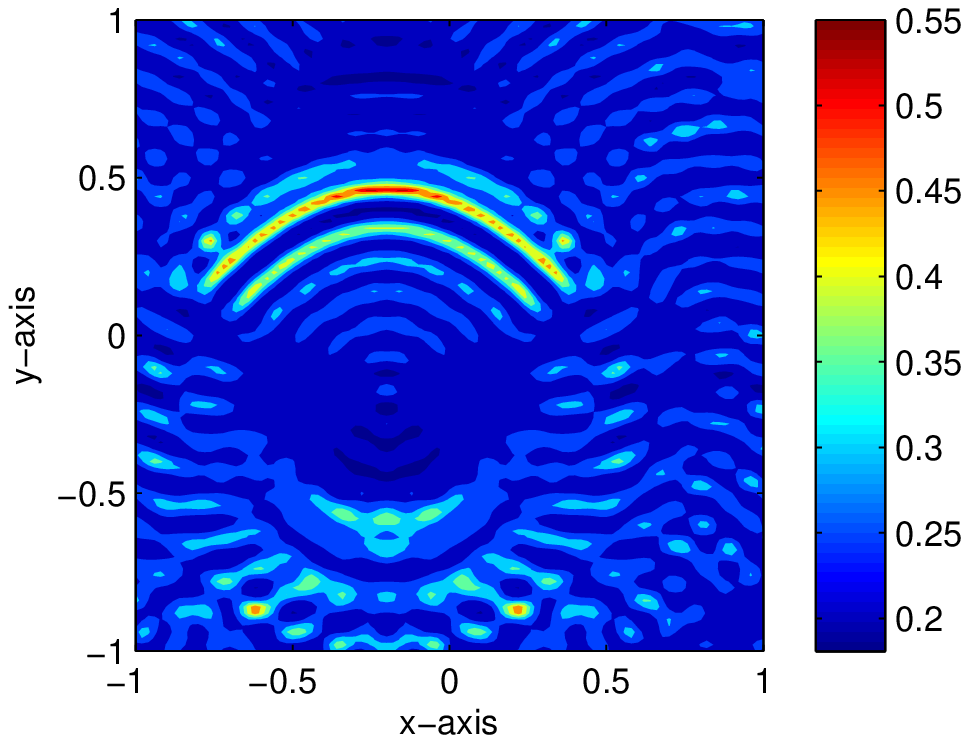}
\caption{Maps of $\mathbb{W}(\mz)$ for $\lambda=\pi$ (top, left), $\lambda=0.8$ (top, right), $\lambda=0.4$ (bottom, left), and $\lambda=0.2$ (bottom, right) when the crack is $\Gamma_1$.}\label{Result1}
\end{center}
\end{figure}

Figure \ref{Result2} shows the maps of $\mathbb{W}(\mz)$ for $\lambda=\pi$, $0.8$, $0.4$, and $0.2$ when the crack is $\Gamma_2$. Similar to the previous result, this result supports the analysis derived in Theorem \ref{TheoremMUSIC}.

\begin{figure}
\begin{center}
\includegraphics[width=0.49\textwidth]{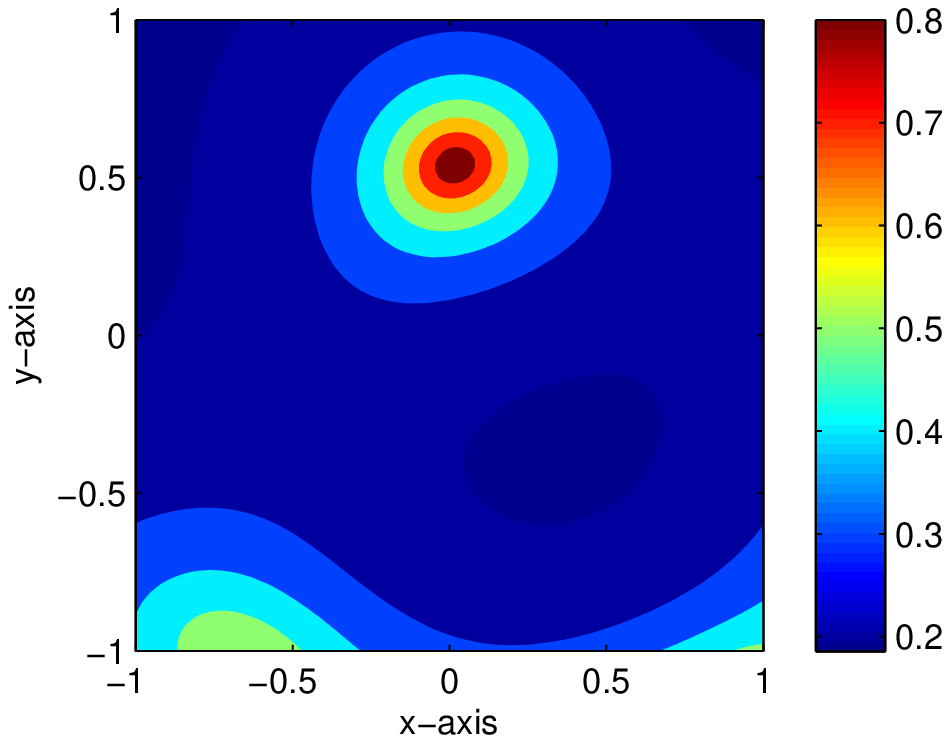}
\includegraphics[width=0.49\textwidth]{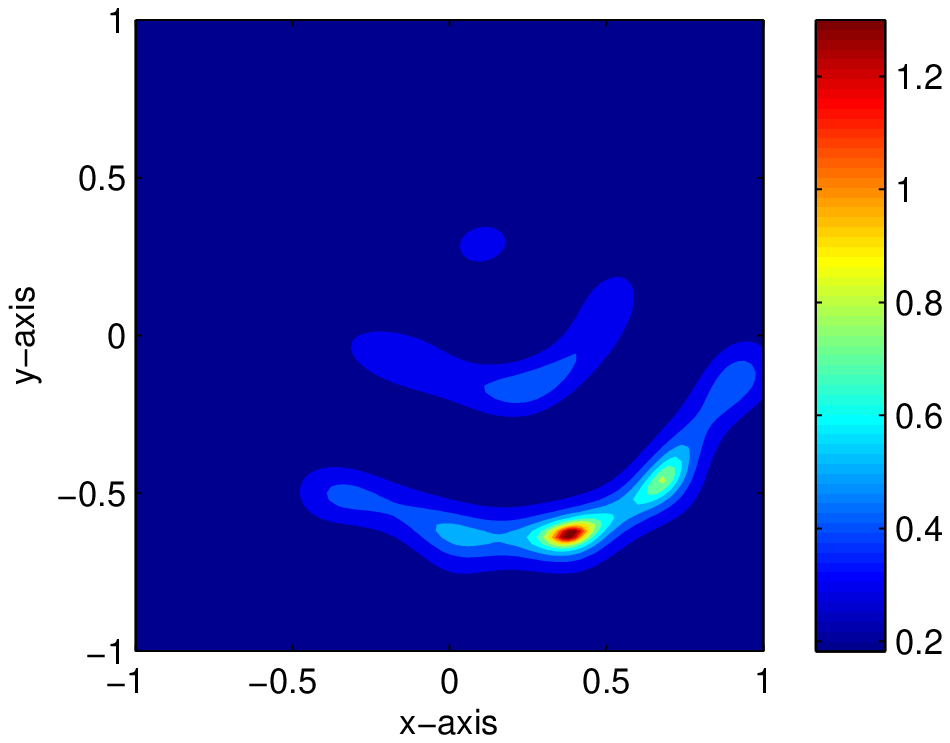}\\
\includegraphics[width=0.49\textwidth]{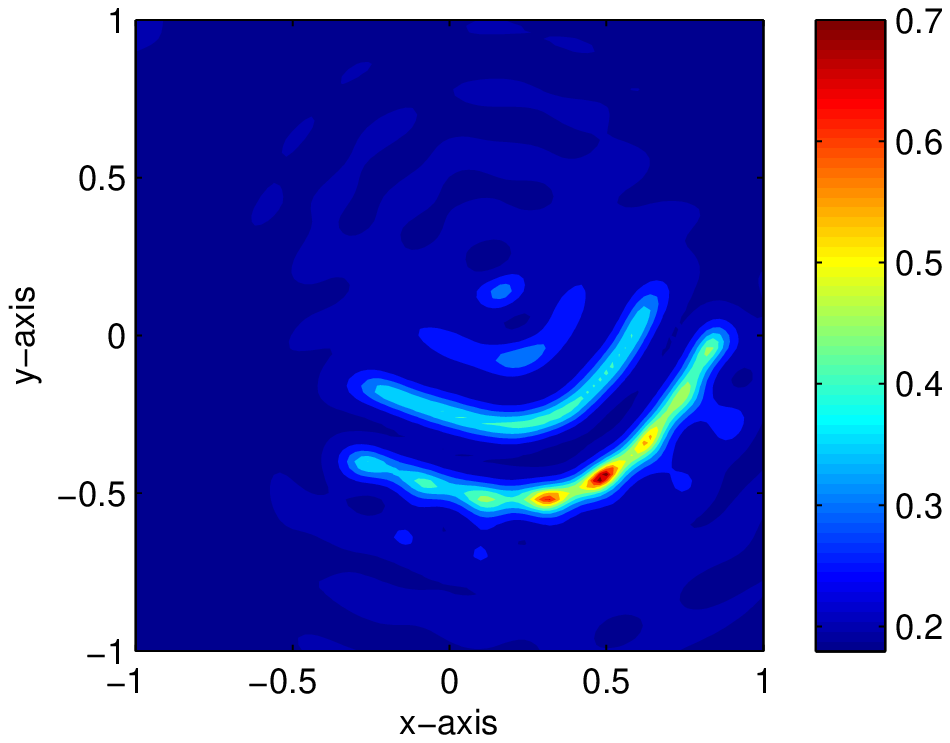}
\includegraphics[width=0.49\textwidth]{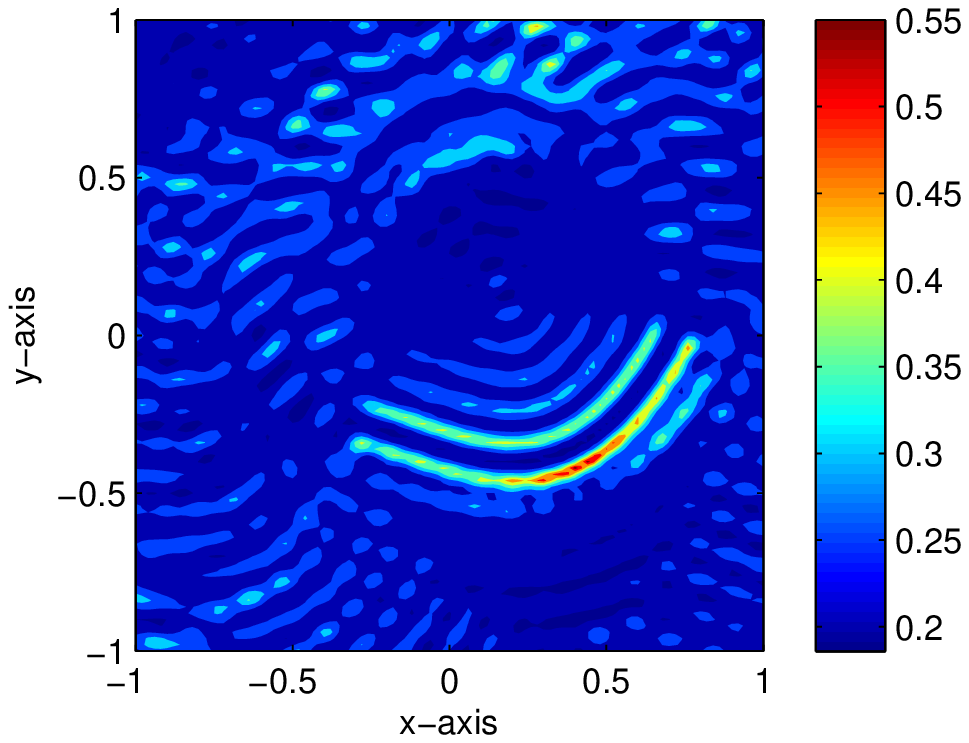}
\caption{Same as Figure \ref{Result1} except the crack is $\Gamma_2$.}\label{Result2}
\end{center}
\end{figure}

It is well-known that the mathematical setting and the numerical analysis could be extended straightforwardly to multiple cracks. For the final example, we consider the imaging of two cracks $\Gamma_1\cup\Gamma_2$ with $N=48$ total directions, refer to Figure \ref{ResultM}. Similar to the previous results, we can recognize two curves along $\Gamma_1$ and $\Gamma_2$ when sufficiently large $k=2\pi/0.2$ is applied. But, it is very hard to identify the existence of cracks with small $k$.

\begin{figure}
\begin{center}
\includegraphics[width=0.49\textwidth]{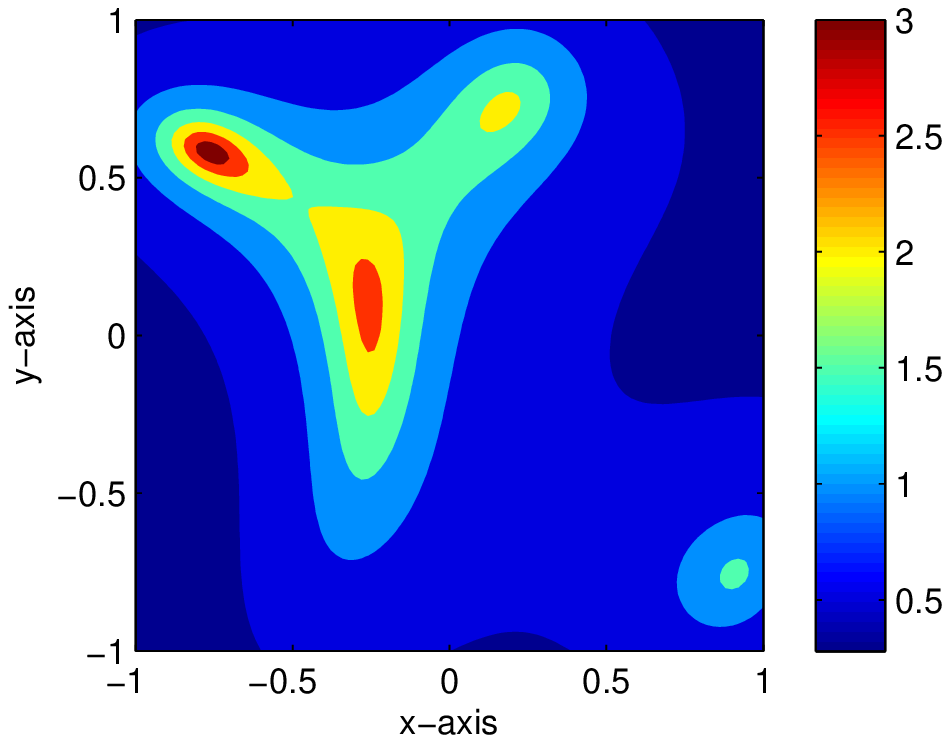}
\includegraphics[width=0.49\textwidth]{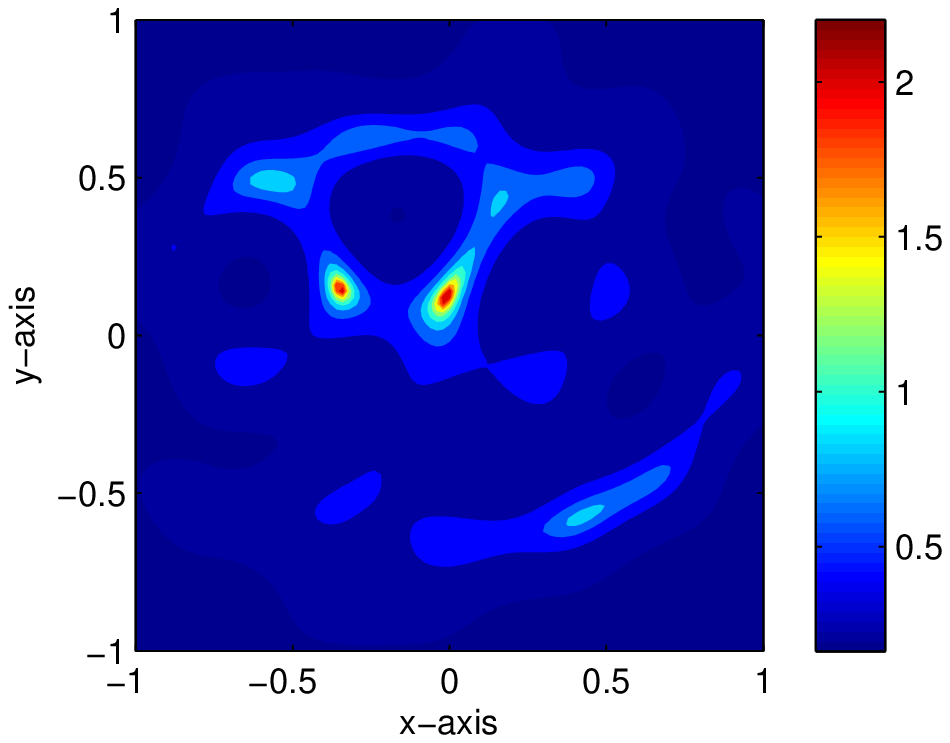}\\
\includegraphics[width=0.49\textwidth]{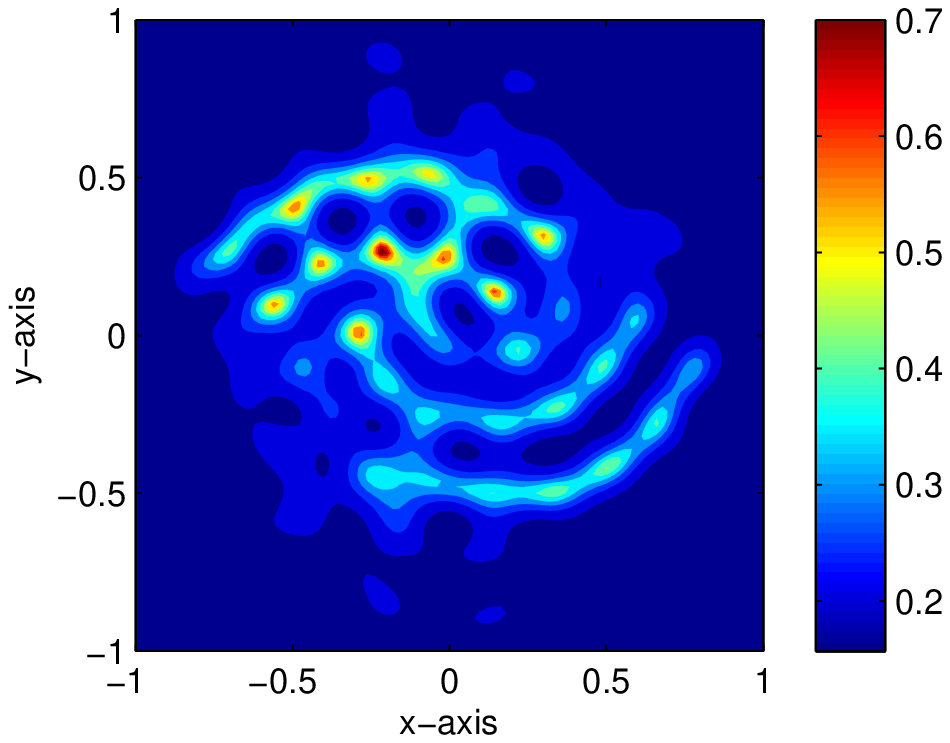}
\includegraphics[width=0.49\textwidth]{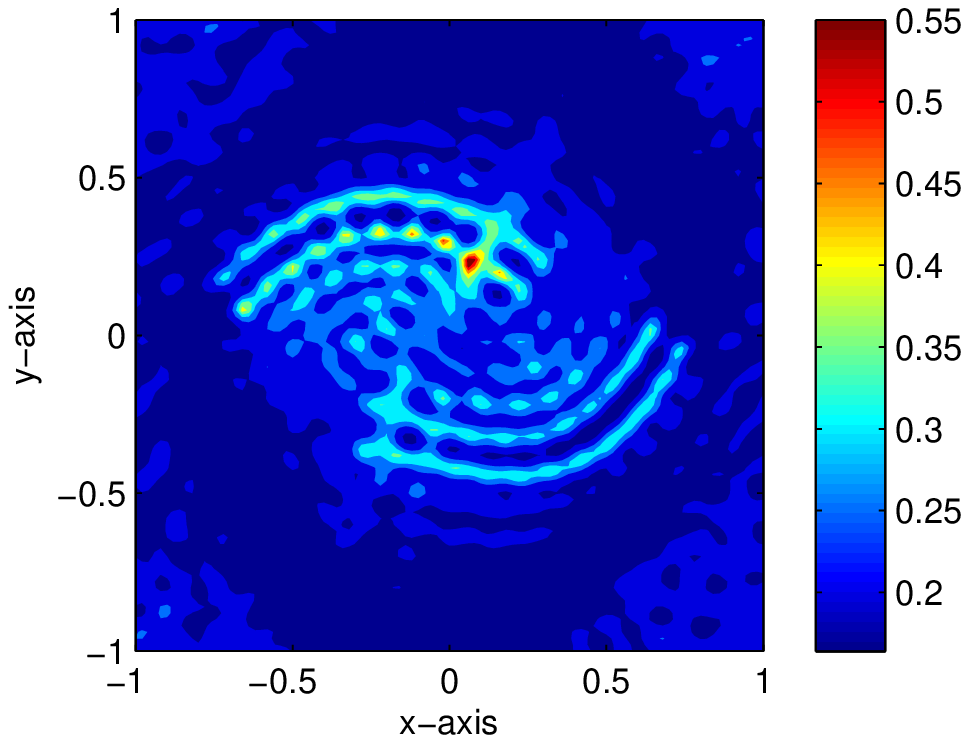}
\caption{Same as Figure \ref{Result1} except the crack is $\Gamma_2$.}\label{ResultM}
\end{center}
\end{figure}

\section{Conclusion}\label{sec:5}
We considered MUSIC-type algorithm for imaging of sound-hard arc. Based on the relationship between MUSIC-type imaging function and Bessel function of order $1$ of the first kind, we examined intrinsic properties and limitation of MUSIC.

Although we considered MUSIC-type imaging in full-view inverse scattering problem, based on our contributions \cite{AJP,JKP}, MUSIC is also applicable to limited-view inverse problem with appropriate condition of the range of incident and observation direction. Discovering certain properties of MUSIC in limited-view problem will be an interesting research topic.

\bibliographystyle{elsarticle-num-names}
\bibliography{../../../References}

\end{document}